\documentclass[12pt]{amsart}

\usepackage{mathpazo}

\usepackage{pifont}

\theoremstyle{plain}
\newtheorem{theorem}{Theorem}[section]

\newtheorem{lemma}[theorem]{Lemma}
\newtheorem{corollary}[theorem]{Corollary}
\newtheorem{claim}[theorem]{Claim}

\theoremstyle{definition}

\theoremstyle{remark}

\makeatletter
\newcommand{\opnorm}{\@ifstar\@opnorms\@opnorm}
\newcommand{\@opnorms}[1]{%
  \left|\mkern-1.5mu\left|\mkern-1.5mu\left|
   #1
  \right|\mkern-1.5mu\right|\mkern-1.5mu\right|
}
\newcommand{\@opnorm}[2][]{%
  \mathopen{#1|\mkern-1.5mu#1|\mkern-1.5mu#1|}
  #2
  \mathclose{#1|\mkern-1.5mu#1|\mkern-1.5mu#1|}
}
\makeatother

\begin{document}

\title[Non-universal families]{Non-universal families of separable Banach spaces}
\author{Ond\v{r}ej Kurka}
\thanks{The research was supported by the grant GA\v{C}R 14-04892P. The author is a junior researcher in the University Centre for Mathematical Modelling, Applied Analysis and Computational Mathematics (MathMAC). The author is a member of the Ne\v{c}as Center for Mathematical Modeling.}
\address{Department of Mathematical Analysis, Charles University, Soko\-lovsk\'a 83, 186 75 Prague 8, Czech Republic}
\email{kurka.ondrej@seznam.cz}
\keywords{Isometrically universal Banach space, Effros-Borel structure, Analytic set, Monotone basis, Strict convexity}
\subjclass[2010]{Primary 46B04, 54H05; Secondary 46B15, 46B20, 46B25}
\begin{abstract}
We prove that if $ \mathcal{C} $ is a family of separable Banach spaces which is analytic with respect to the Effros-Borel structure and none $ X \in \mathcal{C} $ is isometrically universal for all separable Banach spaces, then there exists a separable Banach space with a monotone Schauder basis which is isometrically universal for $ \mathcal{C} $ but still not for all separable Banach spaces. We also establish an analogous result for the class of strictly convex spaces.
\end{abstract}
\maketitle

\section{Introduction and main results}

Let $ \mathcal{C} $ be a class of Banach spaces. We say that a Banach space $ X $ is \emph{isomorphically (isometrically) universal for $ \mathcal{C} $} if it contains an isomorphic (isometric) copy of every member of $ \mathcal{C} $.

The present paper, as well as two author's recent papers \cite{kurka1, kurka2}, establishes isometric counterparts of results concerning universality questions in separable Banach space theory and their natural connection with descriptive set theory (see \cite{bourgain, argyros, bossard1, bossard2, argyrosdodos, dodosferenczi, dodos, dodosquot, godkal, kurka}). These three papers together give a solution of a problem posed by G.~Godefroy \cite{godefroyprobl} if there exists any isometric version of the amalgamation theory of S.~A.~Argyros and P.~Dodos \cite{argyrosdodos} which would provide isometrically universal spaces for small, or regular, isometric classes of Banach spaces.

For a class $ \mathcal{C} $ of separable Banach spaces, it is a natural question whether $ \mathcal{C} $ is \textquotedblleft generic\textquotedblright {} in the sense that every separable Banach space which is isomorphically (isometrically) universal for $ \mathcal{C} $ is also isomorphically (isometrically) universal for all separable Banach spaces.

Employing methods from descriptive set theory, J.~Bourgain \cite{bourgain} strengthened a well-known result of W.~Szlenk \cite{szlenk} and proved that the answer is positive for the class of separable reflexive spaces (in the isomorphic setting). The result was revisited by B.~Bossard \cite{bossard2} who proved that any analytic set of separable Banach spaces (defined below) that contains every separable reflexive space up to isomorphism must also contain an element which is isomorphically universal for all separable Banach spaces.

These two results motivated the authors of \cite{argyrosdodos} to introduce two corresponding concepts of the Bourgain genericity and the Bossard genericity. It is easy to show that a Bossard generic class is Bourgain generic, because the set of all spaces which can be embedded isomorphically into a separable Banach space $ X $ is analytic. The opposite implication, conjectured in \cite{argyrosdodos}, was proved only for classes of spaces with a basis.

To drop the reliance on basis, P.~Dodos \cite{dodos} developed a parameterized version of a construction of $ \mathcal{L}_{\infty} $-spaces due to J.~Bourgain and G.~Pisier \cite{bourgainpisier}. This enabled to prove the equivalence between the Bourgain genericity and the Bossard genericity at last.

In the present work, we find an isometric counterpart of a result from \cite{dodos}. We prove the following theorem.

\begin{theorem} \label{thmmain1}
Let $ \mathcal{C} $ be an analytic set of Banach spaces none of which is isometrically universal for all separable Banach spaces. Then there exists a Banach space $ E $ with a monotone basis which is isometrically universal for $ \mathcal{C} $ but still not for all separable Banach spaces.
\end{theorem}

It follows from this theorem and from \cite[Lemma~7(ii)]{godefroy} that the two considered genericities coincide in the isometric setting as well.

\begin{corollary}
For a class $ \mathcal{P} $ of separable Banach spaces, the following assertions are equivalent:

{\rm (a)} A separable Banach space which is isometrically universal for $ \mathcal{P} $ is also isometrically universal for all separable Banach spaces.

{\rm (b)} Every analytic set $ \mathcal{C} $ of separable Banach spaces containing all members of $ \mathcal{P} $ up to isometry must also contain an element which is isometrically universal for all separable Banach spaces.
\end{corollary}

We do not use the Bourgain-Pisier construction, as P.~Dodos did in \cite{dodos}, which is not surprising, simply because the isomorphic and the isometric universality are quite different notions. Nevertheless, there are still analogies between our methods and methods from \cite{dodos}. The main analogy is that the result has been already proved for classes of spaces with a monotone basis (see \cite[Theorem 1.2]{kurka1}) and so our task is to find an embedding of a general separable space into a space with a monotone basis. The embedding must preserve non-universality and must be simple from the descriptive set theoretic viewpoint. One may notice that there are also analogies with the parameterized version of Zippin's embedding theorem \cite{zippin} due to B.~Bossard \cite{bossard1} (see also \cite[Chapter~5]{dodostopics}).

Let us remark that a more general version of Theorem~\ref{thmmain1} holds (cf.~with \cite[remark~(IV)]{kurka1}). Let $ H $ be a separable Banach space for which there are $ a \in H $ and a subset $ D \subset H $ whose closed linear span contains an isometric copy of $ H $ itself and such that $ \Vert a \pm d \Vert = \Vert a \Vert $ for every $ d \in D $. Then the theorem holds for the class of spaces not containing an isometric copy of $ H $. Among the universal space $ H = C(2^{\mathbb{N}}) $, the property is fulfilled e.g. by the spaces $ H = c_{0} $ and $ H = \ell_{1} $.

The basic property of our embedding is that it creates no new line segments in the unit sphere. For this reason, the method works at the same time for the class of strictly convex spaces.

\begin{theorem} \label{thmmain2}
Let $ \mathcal{C} $ be an analytic set of separable strictly convex Banach spaces. Then there exists a strictly convex Banach space $ E $ with a monotone basis which is isometrically universal for $ \mathcal{C} $.
\end{theorem}

It was proved by E.~Odell and Th.~Schlumprecht \cite{odschl} that there exists a separable reflexive space which is isomorphically universal for separable uniformly convex spaces (actually, there exists an isometrically universal space, see \cite{kurka2}). Since the set of all separable uniformly convex spaces is Borel (see \cite[Corollary~5]{dodosferenczi}), we obtain the following result.

\begin{corollary}
There exists a separable strictly convex Banach space which is isometrically universal for all separable uniformly convex Banach spaces. 
\end{corollary}

\section{Preliminaries}

By a basis we mean a Schauder basis. A basis $ x_{1}, x_{2}, \dots $ is said to be \emph{monotone} if the associated partial sum operators $ P_{n} : \sum_{k=1}^{\infty} a_{k}x_{k} \mapsto \sum_{k=1}^{n} a_{k}x_{k} $ satisfy $ \Vert P_{n} \Vert \leq 1 $.

A \emph{Polish space (topology)} means a separable completely metrizable space (topology). A set $ P $ equipped with a $ \sigma $-algebra is called a \emph{standard Borel space} if the $ \sigma $-algebra is generated by a Polish topology on $ P $. A subset of a standard Borel space is called \emph{analytic} if it is a Borel image of a Polish space.

The following result can be found e.g. in \cite[p.~297]{kechris}.

\begin{theorem}[Arsenin, Kunugui] \label{arskun}
Let $ X $ be a standard Borel space, $ Y $ a Polish space and $ R \subset X \times Y $ a Borel set such that all its sections $ R_{x} = \{ y \in Y : (x, y) \in R \}, x \in X, $ are $ \sigma $-compact. Then the projection $ \pi_{X}(R) $ of $ R $ is Borel and there exists a Borel mapping $ f : \pi_{X}(R) \to Y $ with $ f(x) \in R_{x} $ for every $ x \in \pi_{X}(R) $.
\end{theorem}

For a topological space $ X $, the set $ \mathcal{F}(X) $ of all closed subsets of $ X $ is equipped with the \emph{Effros-Borel structure}, defined as the $ \sigma $-algebra generated by the sets
$$ \{ F \in \mathcal{F}(X) : F \cap U \neq \emptyset \} $$
where $ U $ varies over open subsets of $ X $. If $ X $ is Polish, then, equipped with this $ \sigma $-algebra, $ \mathcal{F}(X) $ forms a standard Borel space.

We will need the following basic fact (see e.g. \cite[p.~76]{kechris}).

\begin{theorem}[Kuratowski, Ryll-Nardzewski] \label{kurryl}
Let $ X $ be a Polish space. Then there exists a sequence $ d_{1}, d_{2}, \dots : \mathcal{F}(X) \to X $ of Borel mappings such that $ d_{1}(F), d_{2}(F), \dots $ is dense in $ F $ for each non-empty $ F \in \mathcal{F}(X) $.
\end{theorem}

The \emph{standard Borel space of subspaces of a separable Banach space $ A $} is defined by
$$ \mathcal{SE}(A) = \big\{ F \in \mathcal{F}(A) : \textrm{$ F $ is linear} \big\} $$
and considered as a subspace of $ \mathcal{F}(A) $.

By an \emph{analytic set of separable Banach spaces} we mean an analytic subset of $ \mathcal{SE}(C([0, 1])) $. It is well known that the spaces $ C([0, 1]) $ and $ C(2^{\mathbb{N}}) $ are isometrically universal for all separable Banach spaces.

At the end of this short section, we recall a classical result from Banach space theory (see e.g. \cite[p.~125]{fhhmpz}).

\begin{theorem}[Banach, Dieudonn\'e] \label{bandie}
Let $ X $ be a Banach space and $ M $ be a convex subset of $ X^{*} $. If $ M \cap nB_{X^{*}} $ is $ w^{*} $-closed for every $ n \in \mathbb{N} $, then $ M $ is $ w^{*} $-closed.
\end{theorem}

\section{First lemma}

\begin{lemma} \label{lemma1}
Let $ A $ be a separable Banach space. Then there exist an isometry $ I : A \to C([0, 1]) $ and a sequence $ p_{1}, p_{2}, \dots $ of Borel functions $ p_{n} : \mathcal{SE}(A) \to [0, 1] $ such that the following properties are valid for every $ X \in \mathcal{SE}(A) $:

{\rm (i)} $ p_{1}(X) = 0 $ and $ p_{2}(X) = 1 $,

{\rm (ii)} $ p_{n}(X) \neq p_{m}(X) $ for $ n \neq m $,

{\rm (iii)} the sequence $ p_{1}(X), p_{2}(X), \dots $ is dense in $ [0, 1] $,

{\rm (iv)} the subspace $ IX $ of $ C([0, 1]) $ is closed in the topology generated by the points $ p_{1}(X), p_{2}(X), \dots $.
\end{lemma}

To provide an isometry $ I $, we follow a standard method. Let $ \varphi : [0, 1] \to (B_{A^{*}}, w^{*}) $ be a continuous surjection and let
$$ (Ia)(t) = \varphi(t)(a), \quad a \in A, \; t \in [0, 1]. $$
In several steps, we show that the choice of $ I $ works and a suitable sequence $ p_{1}, p_{2}, \dots $ of Borel functions exists.

\begin{claim} \label{claim11}
$ IA $ is closed in the pointwise topology.
\end{claim}

\begin{proof}
Assume that $ f \in C([0, 1]) $ belongs to the closure of $ IA $ in the pointwise topology. We show first that there is a function $ h : B_{A^{*}} \to [-\Vert f \Vert, \Vert f \Vert] $ such that
$$ f = h \circ \varphi. $$
We just need to check that $ \varphi(u) = \varphi(v) \Rightarrow f(u) = f(v) $. Given $ u, v \in [0, 1] $ with $ \varphi(u) = \varphi(v) $, we obtain for every $ a \in A $ that $ (Ia)(u) = \varphi(u)(a) = \varphi(v)(a) = (Ia)(v) $. Since $ f $ belongs to the closure of $ IA $ in the pointwise topology, we get $ f(u) = f(v) $.

We have $ h(0) = 0 $. Indeed, choosing $ u \in [0, 1] $ with $ \varphi(u) = 0 $, we obtain $ (Ia)(u) = \varphi(u)(a) = 0 $ for $ a \in A $, and so $ 0 = f(u) = h(\varphi(u)) = h(0) $.

Further, $ h $ is affine. Consider $ a^{*}, b^{*} \in B_{A^{*}} $ and $ \alpha, \beta \in [0, 1] $ with $ \alpha + \beta = 1 $. Choose $ u, v, w \in [0, 1] $ with $ \varphi(u) = a^{*}, \varphi(v) = b^{*} $ and $ \varphi(w) = \alpha a^{*} + \beta b^{*} $. We have
\begin{align*}
(Ia)(w) & = \varphi(w)(a) = (\alpha a^{*} + \beta b^{*})(a) = \alpha a^{*}(a) + \beta b^{*}(a) \\
 & = \alpha \varphi(u)(a) + \beta \varphi(v)(a) = \alpha (Ia)(u) + \beta (Ia)(v)
\end{align*}
for $ a \in A $, and so
\begin{align*}
f(w) & = \alpha f(u) + \beta f(v), \\
h(\varphi(w)) & = \alpha h(\varphi(u)) + \beta h(\varphi(v)), \\
h(\alpha a^{*} + \beta b^{*}) & = \alpha h(a^{*}) + \beta h(b^{*}).
\end{align*}

Finally, $ h $ is $ w^{*} $-continuous. Let $ a^{*}_{1}, a^{*}_{2}, \dots $ be a sequence in $ B_{A^{*}} $ converging to some $ a^{*} $ in the $ w^{*} $-topology. We need to check that $ h(a^{*}_{n}) \to h(a^{*}) $. Assume the opposite. Then there is a subsequence $ a^{*}_{n_{k}} $ such that $ h(a^{*}_{n_{k}}) \to c \neq h(a^{*}) $. For all $ k \in \mathbb{N} $, let us consider $ u_{k} \in [0, 1] $ such that $ \varphi(u_{k}) = a^{*}_{n_{k}} $. There is a subsequence $ u_{k_{l}} $ which converges to some $ u $. Since $ \varphi $ is continuous, we have $ \varphi(u_{k_{l}}) \to \varphi(u) $, and, using $ \varphi(u_{k}) = a^{*}_{n_{k}} \to a^{*} $, we obtain $ \varphi(u) = a^{*} $. Since $ f $ is continuous, we have $ f(u_{k_{l}}) \to f(u) $, and, using $ f(u_{k_{l}}) = h(\varphi(u_{k_{l}})) = h(a^{*}_{n_{k_{l}}}) \to c $, we obtain $ f(u) = c $. Consequently, $ f(u) = c \neq h(a^{*}) = h(\varphi(u)) = f(u) $, which is a contradiction.

We have shown that $ h $ is $ w^{*} $-continuous, affine and $ h(0) = 0 $. So, $ h $ can be extended to a linear functional on $ A^{*} $ which is $ w^{*} $-continuous by Theorem~\ref{bandie}. Consequently, there is $ a \in A $ such that $ h(a^{*}) = a^{*}(a) $ for all $ a^{*} \in B_{A^{*}} $. Then $ f = Ia $, and the proof of the claim is finished.
\end{proof}

\begin{claim} \label{claim12}
There is a sequence of numbers $ v_{1}, v_{2}, \dots $ that is dense in $ [0, 1] $ and generates a topology in which $ IA $ is closed.
\end{claim}

\begin{proof}
By Claim~\ref{claim11}, every $ f \in C([0, 1]) \setminus IA $ has a neighborhood $ U_{f} \subset C([0, 1]) \setminus IA $ of the form
$$ U_{f} = \big\{ g \in C([0, 1]) : |g(u_{k}) - f(u_{k})| < \varepsilon, 1 \leq k \leq n \big\} $$
for some suitable $ \varepsilon, n $ and $ u_{1}, \dots, u_{n} $. Since $ C([0, 1]) \setminus IA $ can be covered by countably many such neighborhoods, it is possible to collect all rational numbers and all numbers $ u_{k} $ associated to the members of this covering.
\end{proof}

\begin{claim} \label{claim13}
For every open ball $ U \subset A $, the set
$$ R = \big\{ (X, a^{*}) \in \mathcal{SE}(A) \times B_{A^{*}} : a^{*}(x) = 0, x \in X, a^{*}(x) > 0, x \in U \big\} $$
is Borel in $ \mathcal{SE}(A) \times (B_{A^{*}}, w^{*}) $ and all its sections $ R_{X}, X \in \mathcal{SE}(A), $ are $ \sigma $-compact.
\end{claim}

\begin{proof}
Let $ a $ be the center of $ U $. Since every $ a^{*} \neq 0 $ maps open balls onto open intervals, we have
$$ (X, a^{*}) \in R \quad \Leftrightarrow \quad a^{*}(x) = 0, x \in X, a^{*}(x) \geq 0, x \in U, a^{*}(a) > 0. $$
The dual unit ball $ B_{A^{*}} $ is compact in the $ w^{*} $-topology. As $ a^{*}(a) > 0 $ if and only if $ a^{*}(a) \geq 1/j $ for some $ j \in \mathbb{N} $, it follows that the sections $ R_{X}, X \in \mathcal{SE}(A), $ are $ \sigma $-compact.

Let $ x_{1}, x_{2}, \dots $ be dense in $ U $ and let $ d_{1}, d_{2}, \dots : \mathcal{F}(A) \to A $ be Borel selectors provided by Theorem~\ref{kurryl}. Using the equivalence
$$ (X, a^{*}) \in R \quad \Leftrightarrow \quad a^{*}(d_{n}(X)) = 0, a^{*}(x_{n}) \geq 0, n \in \mathbb{N}, a^{*}(a) > 0 $$
and the continuity of $ (a^{*}, x) \in (B_{A^{*}}, w^{*}) \times A \mapsto a^{*}(x) $, we see that $ R $ is Borel.
\end{proof}

\begin{claim} \label{claim14}
There is a sequence $ s_{1}, s_{2}, \dots $ of Borel mappings $ s_{n} : \mathcal{SE}(A) \to (B_{A^{*}}, w^{*}) $ such that, for every $ X \in \mathcal{SE}(A) $ and $ a \in A \setminus X $, there is $ n \in \mathbb{N} $ with $ s_{n}(X)(a) \neq 0 $ and $ s_{n}(X)(x) = 0 $ for all $ x \in X $.
\end{claim}

\begin{proof}
Let $ U_{1}, U_{2}, \dots $ be a countable basis of the norm topology of $ A $ consisting of open balls. For each $ n \in \mathbb{N} $, let us consider a set
$$ R_{n} = \big\{ (X, a^{*}) \in \mathcal{SE}(A) \times B_{A^{*}} : a^{*}(x) = 0, x \in X, a^{*}(x) > 0, x \in U_{n} \big\}. $$
By Claim~\ref{claim13} and Theorem~\ref{arskun}, there exists a Borel mapping $ s_{n} : \mathcal{SE}(A) \to (B_{A^{*}}, w^{*}) $ such that $ (R_{n})_{X} \neq \emptyset \Rightarrow s_{n}(X) \in (R_{n})_{X} $ (we consider a Borel extension of the mapping provided by Theorem~\ref{arskun}). Let us check that the mappings $ s_{n} $ work.

Let $ X \in \mathcal{SE}(A) $ and $ a \in A \setminus X $. For some $ n \in \mathbb{N} $, we have $ a \in U_{n} \subset A \setminus X $. By the Hahn-Banach theorem, there exists $ a^{*} \in B_{A^{*}} $ such that $ a^{*}(x) = 0 $ for $ x \in X $ and $ a^{*}(x) > 0 $ for $ x \in U_{n} $. It means that $ (R_{n})_{X} $ is non-empty, and so $ s_{n}(X) \in (R_{n})_{X} $. Therefore, $ s_{n}(X) $ has the desired property.
\end{proof}

\begin{claim} \label{claim15}
There is a sequence $ p_{1}, p_{2}, \dots $ of Borel functions $ p_{n} : \mathcal{SE}(A) \to [0, 1] $ such that {\rm (iii)} and {\rm (iv)} are valid for every $ X \in \mathcal{SE}(A) $.
\end{claim}

\begin{proof}
Let a sequence of numbers $ v_{1}, v_{2}, \dots $ be given by Claim~\ref{claim12} and a sequence of Borel mappings $ s_{1}, s_{2}, \dots $ be given by Claim~\ref{claim14}. Let
$$ q_{n}(X) = \min \varphi^{-1}(s_{n}(X)), \quad n \in \mathbb{N}, \; X \in \mathcal{SE}(A). $$
Let us show that the function $ q_{n} $ is Borel for every $ n \in \mathbb{N} $. It is sufficient to check that the mapping $ a^{*} \mapsto \min \varphi^{-1}(a^{*}) $ is Borel from $ (B_{A^{*}}, w^{*}) $ into $ [0, 1] $. For $ u \in [0, 1] $, the set $ \{ a^{*} \in B_{A^{*}} : \min \varphi^{-1}(a^{*}) \leq u \} = \varphi([0, u]) $ is compact, and thus Borel.

Further, let us show that, for every $ X \in \mathcal{SE}(A) $, the subspace $ IX $ is closed in the topology generated by the points $ q_{1}(X), q_{2}(X), \dots $ and $ v_{1}, v_{2}, \dots $. Given $ f \in C([0, 1]) \setminus IX $, there are two possibilities. If $ f \notin IA $, then a property of the sequence $ v_{1}, v_{2}, \dots $ guarantees that $ f $ does not belong to the closure of $ IA $ (and of $ IX $ in particular) in the considered topology. If $ f \in IA $, then we choose $ a \in A $ with $ Ia = f $. Necessarily, $ a \notin X $, and thus there is $ n \in \mathbb{N} $ with $ s_{n}(X)(a) \neq 0 $ and $ s_{n}(X)(x) = 0 $ for all $ x \in X $. We have
$$ f(q_{n}(X)) = (Ia)(q_{n}(X)) = \varphi(q_{n}(X))(a) = s_{n}(X)(a) \neq 0, $$
while
$$ (Ix)(q_{n}(X)) = \varphi(q_{n}(X))(x) = s_{n}(X)(x) = 0, \quad x \in X. $$
Hence, $ f $ does belong to the closure of $ IX $ in the considered topology.

Now, if we define
$$ p_{2n-1}(X) = q_{n}(X), \quad p_{2n}(X) = v_{n}, \quad n \in \mathbb{N}, \; X \in \mathcal{SE}(A), $$
then, for every $ X \in \mathcal{SE}(A) $, conditions (iii) and (iv) are valid.
\end{proof}

To finish the proof of Lemma~\ref{lemma1}, it remains to realize that a sequence from Claim~\ref{claim15} can be modified to satisfy (i) and (ii) as well. If we define $ p'_{1}(X) = 0, p'_{2}(X) = 1 $ and $ p'_{n}(X) = p_{m}(X) $ where $ m $ is the least natural number such that $ | \{ 0, 1, p_{1}(X), \dots, p_{m}(X) \} | = n $, it is easy to check that these functions are Borel and satisfy (i)--(iv).

\section{Second lemma}

\begin{lemma} \label{lemma2}
Let $ A $ be a separable Banach space. Then there exist a separable Banach space $ Z $, an isometry $ J : A \to Z $, a collection $ \{ \Vert \cdot \Vert^{X} : X \in \mathcal{SE}(A) \} $ of norms on $ Z $ and a system $ \{ Q^{X}_{n} : X \in \mathcal{SE}(A), n \in \mathbb{N} \} $ of projections on $ Z $ such that the following properties are valid for every $ X \in \mathcal{SE}(A) $:

{\rm (I)} $ \Vert z \Vert \leq \Vert z \Vert^{X} \leq 2 \Vert z \Vert $ for $ z \in Z $,

{\rm (II)} $ \Vert z \Vert = \Vert z \Vert^{X} $ if and only if $ z \in JX $,

{\rm (III)} the sequence $ Q^{X}_{1}, Q^{X}_{2}, \dots $ is the sequence of partial sum operators associated with a basis of $ Z $ which is monotone in the sense that $ \Vert Q^{X}_{n} \Vert \leq 1 $ and $ \Vert Q^{X}_{n} \Vert^{X} \leq 1 $.

Moreover,

{\rm (IV)} the mapping $ (X, z) \mapsto Q^{X}_{n}z $ is Borel from $ \mathcal{SE}(A) \times Z $ into $ Z $ for every $ n \in \mathbb{N} $,

{\rm (V)} the function $ (X, z) \mapsto \Vert z \Vert^{X} $ is Borel from $ \mathcal{SE}(A) \times Z $ into $ \mathbb{R} $.
\end{lemma}

The proof of this lemma consists of four parts. There are some analogies with a construction by Ghoussoub, Maurey and Schachermayer \cite{gms} and its parameterized version by Bossard \cite{bossard1} (see also \cite[Chapter~5]{dodostopics}).

\ding{182}
Let us fix an isometry $ I : A \to C([0, 1]) $ and a sequence $ p_{1}, p_{2}, \dots $ of Borel functions $ p_{n} : \mathcal{SE}(A) \to [0, 1] $ satisfying properties (i)--(iv) from Lemma~\ref{lemma1}. For every $ X \in \mathcal{SE}(A) $, we define a sequence of projections $ P^{X}_{n} : C([0, 1]) \to C([0, 1]) $. Let $ (P^{X}_{1}f)(t) = f(0) $ for every $ f \in C([0, 1]) $ and $ t \in [0, 1] $. Given $ n \geq 2 $ and $ f \in C([0, 1]) $, let $ P^{X}_{n}f $ be the piecewise linear function which has the same values in $ p_{1}(X), \dots, p_{n}(X) $ as $ f $ and is linear elsewhere.

Using properties (i), (ii) and (iii), we can easily verify the following properties:
\begin{itemize}
\item $ \Vert P^{X}_{n}f \Vert \leq \Vert f \Vert $,
\item $ P^{X}_{n}P^{X}_{m} = P^{X}_{m}P^{X}_{n} = P^{X}_{\min \{ m, n \} } $,
\item $ P^{X}_{n} C([0, 1]) $ has dimension $ n $,
\item $ P^{X}_{n}f \to f $ as $ n \to \infty$.
\end{itemize}
Due to these properties, the sequence $ P^{X}_{1}, P^{X}_{2}, \dots $ is the sequence of partial sum operators associated with a monotone basis of $ C([0, 1]) $.

\begin{claim} \label{claim21}
For every $ n \in \mathbb{N} $, the mapping $ (X, f) \mapsto P^{X}_{n}f $ is Borel from $ \mathcal{SE}(A) \times C([0, 1]) $ into $ C([0, 1]) $.
\end{claim}

\begin{proof}
For $ n \leq 2 $, the assertion is clear, as $ P^{X}_{n} $ does not depend on $ X $. For $ n \geq 3 $, the considered mapping is the composition of the Borel mapping $ (X, f) \mapsto (p_{3}(X), \dots, p_{n}(X), f) $ and the continuous mapping which maps $ (x_{3}, \dots, x_{n}, f) $ to the piecewise linear function which has the same values in $ 0, 1, x_{3}, \dots, x_{n} $ as $ f $ and is linear elsewhere.
\end{proof}

\begin{claim} \label{claim22}
Let $ X \in \mathcal{SE}(A) $ and $ f \in C([0, 1]) $. If $ f \notin IX $, then there is $ n \in \mathbb{N} $ such that $ P^{X}_{n}f \notin P^{X}_{n}IX $.
\end{claim}

\begin{proof}
Property (iv) provides a neighborhood $ W \subset C([0, 1]) \setminus IX $ of $ f $ of the form
$$ W = \big\{ g \in C([0, 1]) : |g(p_{k}(X)) - f(p_{k}(X))| < \varepsilon, 1 \leq k \leq n \big\} $$
for a large enough $ n \in \mathbb{N} $ and a small enough $ \varepsilon > 0 $. It means that every $ g \in IX $ satisfies $ |g(p_{k}(X)) - f(p_{k}(X))| \geq \varepsilon $ for some $ 1 \leq k \leq n $. Since $ (P^{X}_{n}g)(p_{k}(X)) = g(p_{k}(X)) $ and $ (P^{X}_{n}f)(p_{k}(X)) = f(p_{k}(X)) $, it follows that $ P^{X}_{n}g \neq P^{X}_{n}f $.
\end{proof}

\ding{183}
For every $ X \in \mathcal{SE}(A) $, we define a sequence of projections $ Q^{X}_{i} : \ell_{2}(C([0, 1])) \to \ell_{2}(C([0, 1])) $. For $ \mathbf{f} = (f_{1}, f_{2}, \dots) \in \ell_{2}(C([0, 1])) $, let
\begin{align*}
Q^{X}_{1}\mathbf{f} & = (P^{X}_{1}f_{1}, 0, 0, 0, \dots), \\
Q^{X}_{2}\mathbf{f} & = (P^{X}_{2}f_{1}, 0, 0, 0, \dots), \\
Q^{X}_{3}\mathbf{f} & = (P^{X}_{2}f_{1}, P^{X}_{1}f_{2}, 0, 0, \dots), \\
Q^{X}_{4}\mathbf{f} & = (P^{X}_{3}f_{1}, P^{X}_{1}f_{2}, 0, 0, \dots), \\
Q^{X}_{5}\mathbf{f} & = (P^{X}_{3}f_{1}, P^{X}_{2}f_{2}, 0, 0, \dots), \\
Q^{X}_{6}\mathbf{f} & = (P^{X}_{3}f_{1}, P^{X}_{2}f_{2}, P^{X}_{1}f_{3}, 0, \dots), \\
Q^{X}_{7}\mathbf{f} & = (P^{X}_{4}f_{1}, P^{X}_{2}f_{2}, P^{X}_{1}f_{3}, 0, \dots),
\end{align*}
and so on. In this way, $ Q^{X}_{1}, Q^{X}_{2}, \dots $ is the sequence of partial sum operators associated with a monotone basis of $ \ell_{2}(C([0, 1])) $.

Further, let us define an isometry
$$ U : C([0, 1]) \to \ell_{2}(C([0, 1])), \quad f \mapsto \frac{\sqrt{3}}{2} \Big( f, \frac{1}{2}f, \frac{1}{4}f, \dots \Big). $$
This is an isometry indeed, as
$$ \Vert Uf \Vert^{2} = \frac{3}{4} \Big( \Vert f \Vert^{2} + \frac{1}{4} \Vert f \Vert^{2} + \dots \Big) = \Vert f \Vert^{2}, \quad f \in C([0, 1]). $$

\begin{claim} \label{claim23}
For every $ i \in \mathbb{N} $, the mapping $ (X, \mathbf{f}) \mapsto Q^{X}_{i}\mathbf{f} $ is Borel from $ \mathcal{SE}(A) \times \ell_{2}(C([0, 1])) $ into $ \ell_{2}(C([0, 1])) $.
\end{claim}

\begin{proof}
This follows from Claim~\ref{claim21} and the definition of $ Q^{X}_{i} $.
\end{proof}

\begin{claim} \label{claim24}
Let $ X \in \mathcal{SE}(A) $ and $ \mathbf{f} \in \ell_{2}(C([0, 1])) $. If $ \mathbf{f} \notin UIX $, then there is $ i \in \mathbb{N} $ such that $ Q^{X}_{i}\mathbf{f} \notin Q^{X}_{i}UIX $.
\end{claim}

\begin{proof}
For a general $ \mathbf{g} \in \ell_{2}(C([0, 1])) $, we will denote its coordinates by $ g_{k} $ or by $ (\mathbf{g})_{k} $. There are two possibilities.

(1) Assume that $ f_{k} \neq 2f_{k+1} $ for some $ k \in \mathbb{N} $. Let $ n \in \mathbb{N} $ be such that $ P^{X}_{n}f_{k} \neq 2P^{X}_{n}f_{k+1} $ and $ i \in \mathbb{N} $ be large enough so that $ n_{k} \geq n $ and $ n_{k+1} \geq n $ in the expression $ Q^{X}_{i}\mathbf{g} = (P^{X}_{n_{j}}g_{j})_{j=1}^{\infty} $ (we consider $ P^{X}_{0} = 0 $). Then
\begin{align*}
P^{X}_{n}[(Q^{X}_{i}\mathbf{f})_{k}] & = P^{X}_{n}P^{X}_{n_{k}}f_{k} = P^{X}_{n}f_{k} \\
 & \neq 2P^{X}_{n}f_{k+1} = 2P^{X}_{n}P^{X}_{n_{k+1}}f_{k+1} = 2P^{X}_{n}[(Q^{X}_{i}\mathbf{f})_{k+1}],
\end{align*}
while, for every $ g \in C([0, 1]) $ (and, in particular, for every $ g \in IX $),
\begin{align*}
P^{X}_{n} & [(Q^{X}_{i}Ug)_{k}] = P^{X}_{n}P^{X}_{n_{k}}[(Ug)_{k}] = P^{X}_{n}[(Ug)_{k}]= \frac{\sqrt{3}}{2^{k}} P^{X}_{n}g \\
 & = 2P^{X}_{n}[(Ug)_{k+1}] = 2P^{X}_{n}P^{X}_{n_{k+1}}[(Ug)_{k+1}] = 2P^{X}_{n}[(Q^{X}_{i}Ug)_{k+1}],
\end{align*}
and so $ Q^{X}_{i}Ug \neq Q^{X}_{i}\mathbf{f} $.

(2) Assume that $ f_{k} = 2f_{k+1} $ for all $ k \in \mathbb{N} $. Then $ \mathbf{f} = Uf $ for $ f = (2/\sqrt{3})f_{1} $. By our assumption, $ f \notin IX $, and Claim~\ref{claim22} provides $ n \in \mathbb{N} $ such that $ P^{X}_{n}f \notin P^{X}_{n}IX $. Let $ i \in \mathbb{N} $ be large enough so that $ n_{1} \geq n $ in the expression $ Q^{X}_{i}\mathbf{g} = (P^{X}_{n_{j}}g_{j})_{j=1}^{\infty} $ (we consider $ P^{X}_{0} = 0 $). Then, for every $ g \in IX $,
\begin{align*}
P^{X}_{n}[(Q^{X}_{i}Ug)_{1}] & = P^{X}_{n}P^{X}_{n_{1}}[(Ug)_{1}] = P^{X}_{n}[(Ug)_{1}] = \frac{\sqrt{3}}{2} P^{X}_{n}g \\
 & \neq \frac{\sqrt{3}}{2} P^{X}_{n}f = P^{X}_{n}f_{1} = P^{X}_{n}P^{X}_{n_{1}}f_{1} = P^{X}_{n}[(Q^{X}_{i}\mathbf{f})_{1}],
\end{align*}
and so $ Q^{X}_{i}Ug \neq Q^{X}_{i}\mathbf{f} $.
\end{proof}

\begin{claim} \label{claim25}
For $ X \in \mathcal{SE}(A) $ and $ i \in \mathbb{N} $, we have $ \Vert Q^{X}_{i}U \Vert < 1 $.
\end{claim}

\begin{proof}
There is $ k \in \mathbb{N} $ such that points from the range of $ Q^{X}_{i} $ are supported by the first $ k $ coordinates. Given $ f \in C([0, 1]) $, we have
\begin{align*}
\Vert Q^{X}_{i}Uf \Vert^{2} & = \Big(\frac{\sqrt{3}}{2}\Big)^{2} \Big\Vert \Big( P_{n_{1}}f, \frac{1}{2} P_{n_{2}}f, \dots, \frac{1}{2^{k-1}} P_{n_{k}}f, 0, 0, \dots \Big) \Big\Vert^{2} \\
 & = \frac{3}{4} \Big( \Vert P_{n_{1}}f \Vert^{2} + \frac{1}{4} \Vert P_{n_{2}}f \Vert^{2} + \dots + \frac{1}{4^{k-1}} \Vert P_{n_{k}}f \Vert^{2} \Big) \\
 & \leq \frac{3}{4} \Big( \Vert f \Vert^{2} + \frac{1}{4} \Vert f \Vert^{2} + \dots + \frac{1}{4^{k-1}} \Vert f \Vert^{2} \Big) = \Big(1 - \frac{1}{4^{k}}\Big) \Vert f \Vert^{2}
\end{align*}
for some $ n_{1}, \dots, n_{k} \in \mathbb{N} \cup \{ 0 \} $. It follows that $ \Vert Q^{X}_{i}U \Vert^{2} \leq 1 - \frac{1}{4^{k}} $.
\end{proof}

\ding{184}
For every $ X \in \mathcal{SE}(A) $, we define
$$ \Omega^{X} = \overline{\mathrm{co}} \bigg( \frac{1}{2} B_{\ell_{2}(C([0, 1]))} \cup \bigcup_{i=1}^{\infty} Q^{X}_{i}UIB_{X} \bigg). $$
Let us notice that $ UIB_{X} \subset \Omega^{X} $ and $ Q^{X}_{i}\Omega^{X} \subset \Omega^{X} $ for every $ i \in \mathbb{N} $.

\begin{claim} \label{claim26}
The set
$$ \big\{ (X, \mathbf{f}) \in \mathcal{SE}(A) \times \ell_{2}(C([0, 1])) : \mathbf{f} \in \Omega^{X} \big\} $$
is Borel.
\end{claim}

\begin{proof}
Let $ \mathbf{f}_{1}, \mathbf{f}_{2}, \dots $ be dense in $ B_{\ell_{2}(C([0, 1]))} $. Let $ x_{1}, x_{2}, \dots : \mathcal{SE}(A) \to B_{A} $ be Borel mappings such that $ x_{1}(X), x_{2}(X), \dots $ is dense in $ B_{X} $ for every $ X \in \mathcal{SE}(A) $ (it is easy to find such sequence using Theorem~\ref{kurryl}). We have
\begin{eqnarray*}
\mathbf{f} \in \Omega^{X} & \Leftrightarrow & \forall l \in \mathbb{N} \; \exists m \in \mathbb{N} \; \exists k, n_{1}, \dots, n_{m} \in \mathbb{N} \\
 & & \exists \gamma_{0}, \gamma_{1}, \dots, \gamma_{m} \in \mathbb{Q} \cap [0, 1], \; \sum_{i=0}^{m} \gamma_{i} = 1 : \\
 & & \bigg\Vert \mathbf{f} - \Big[ \frac{1}{2} \gamma_{0} \mathbf{f}_{k} + \sum_{i=1}^{m} \gamma_{i} Q^{X}_{i}UIx_{n_{i}}(X) \Big] \bigg\Vert < \frac{1}{l}.
\end{eqnarray*}
It remains to note that, by Claim~\ref{claim23}, the mapping $ X \mapsto Q^{X}_{i}UIx_{n}(X) $ is Borel for all $ i, n \in \mathbb{N} $.
\end{proof}

\begin{claim} \label{claim27}
For $ X \in \mathcal{SE}(A) $ and $ \mathbf{f} \in \ell_{2}(C([0, 1])) \setminus UIX $ with $ \Vert \mathbf{f} \Vert = 1 $, we have $ \mathbf{f} \notin \Omega^{X} $.
\end{claim}

\begin{proof}
Claim~\ref{claim24} provides $ i \in \mathbb{N} $ such that $ Q^{X}_{i}\mathbf{f} \notin Q^{X}_{i}UIX $. Let $ \mathbf{f}^{*} \in \ell_{2}(C([0, 1]))^{*} $ be such that $ \Vert \mathbf{f}^{*} \Vert = 1 = \mathbf{f}^{*}(\mathbf{f}) $ and let $ z^{*} $ be a functional on $ Q^{X}_{i}\ell_{2}(C([0, 1])) $ such that $ \Vert z^{*} \Vert = 1 $, $ z^{*}(Q^{X}_{i}\mathbf{f}) > 0 $ and $ z^{*}(Q^{X}_{i}\mathbf{g}) = 0 $ for $ \mathbf{g} \in UIX $. By Claim~\ref{claim25}, there is $ \varepsilon \in (0, 1] $ such that $ \Vert Q^{X}_{j}U \Vert \leq 1 - \varepsilon $ for $ 1 \leq j < i $. Let us define
$$ \mathbf{g}^{*} = \mathbf{f}^{*} + \varepsilon \cdot z^{*} \circ Q^{X}_{i}. $$
Then
$$ \mathbf{g}^{*}(\mathbf{f}) = \mathbf{f}^{*}(\mathbf{f}) + \varepsilon \cdot z^{*}(Q^{X}_{i}\mathbf{f}) > 1. $$
We claim that $ \mathbf{g}^{*} $ separates $ \mathbf{f} $ from $ \Omega^{X} $, showing that
$$ \mathbf{g}^{*}(\mathbf{u}) \leq 1, \quad \mathbf{u} \in \Omega^{X}. $$
If $ \mathbf{u} \in \frac{1}{2} B_{\ell_{2}(C([0, 1]))} $, then $ \mathbf{g}^{*}(\mathbf{u}) \leq \Vert \mathbf{g}^{*} \Vert \Vert \mathbf{u} \Vert \leq (1 + \varepsilon) \cdot \frac{1}{2} \leq 1 $. So, it remains to show that $ \mathbf{g}^{*}(\mathbf{u}) \leq 1 $ for $ \mathbf{u} = Q^{X}_{j}Ug $ where $ j \in \mathbb{N} $ and $ g \in IB_{X} $. If $ 1 \leq j < i $, then $ \mathbf{g}^{*}(\mathbf{u}) \leq \Vert \mathbf{g}^{*} \Vert \Vert Q^{X}_{j}U \Vert \Vert g \Vert \leq (1 + \varepsilon) (1 - \varepsilon) \leq 1 $. If $ j \geq i $, then $ z^{*}(Q^{X}_{i}\mathbf{u}) = z^{*}(Q^{X}_{i}Q^{X}_{j}Ug) = z^{*}(Q^{X}_{i}Ug) = 0 $ and
$$ \mathbf{g}^{*}(\mathbf{u}) = \mathbf{f}^{*}(\mathbf{u}) + \varepsilon \cdot z^{*}(Q^{X}_{i}\mathbf{u}) = \mathbf{f}^{*}(\mathbf{u}) \leq \Vert \mathbf{f}^{*} \Vert \Vert Q^{X}_{j} \Vert \Vert Ug \Vert \leq 1, $$
which completes the verification of $ \mathbf{f} \notin \Omega^{X} $.
\end{proof}

\ding{185}
Now, we are ready to finish the proof of Lemma~\ref{lemma2}. For every $ X \in \mathcal{SE}(A) $, we define $ \Vert \cdot \Vert^{X} $ as the norm on $ \ell_{2}(C([0, 1])) $ which has $ \Omega^{X} $ for its unit ball. Let us check that properties (I)--(V) are valid for the choice $ Z = \ell_{2}(C([0, 1])) $ and $ J = UI $.

(I) This follows from $ \frac{1}{2} B_{\ell_{2}(C([0, 1]))} \subset \Omega^{X} \subset B_{\ell_{2}(C([0, 1]))} $.

(II) We know that $ UIB_{X} \subset \Omega^{X} \subset B_{\ell_{2}(C([0, 1]))} $, which means that $ \Vert \mathbf{f} \Vert = \Vert \mathbf{f} \Vert^{X} $ for every $ \mathbf{f} \in UIX $. Assume that $ \mathbf{f} \in \ell_{2}(C([0, 1])) \setminus UIX $. Assume moreover without loss of generality that $ \Vert \mathbf{f} \Vert = 1 $. By Claim~\ref{claim27}, we have $ \mathbf{f} \notin \Omega^{X} $, which means that $ \Vert \mathbf{f} \Vert^{X} > 1 = \Vert \mathbf{f} \Vert $.

(III) It follows from $ Q^{X}_{i}\Omega^{X} \subset \Omega^{X} $ that $ \Vert Q^{X}_{i} \Vert^{X} \leq 1 $.

(IV) This is already provided by Claim~\ref{claim23}.

(V) By Claim~\ref{claim26}, the pre-image of $ [0, 1] $ is Borel. Clearly, the pre-image of $ [0, r] $ is also Borel for every $ r > 0 $, which gives (V).

\section{Proof of main results}

Let us consider $ A = C([0, 1]) $. Let a separable Banach space $ Z $, an isometry $ J : C([0, 1]) \to Z $, a collection $ \{ \Vert \cdot \Vert^{X} : X \in \mathcal{SE}(C([0, 1])) \} $ of norms on $ Z $ and a system $ \{ Q^{X}_{n} : X \in \mathcal{SE}(C([0, 1])), n \in \mathbb{N} \} $ of projections on $ Z $ satisfy properties (I)--(V) from Lemma~\ref{lemma2}.

We are going to apply the same technique as in \cite[Section~8]{kurka1}. This will enable to obtain a new collection $ \{ \opnorm{\cdot}^{X} : X \in \mathcal{SE}(C([0, 1])) \} $ of norms on $ Z $ with the same properties and with the additional property that all line segments contained in the unit sphere of $ (Z, \opnorm{\cdot}^{X}) $ are contained in $ JX $. 

Let $ \varrho $ be a norm on $ \mathbb{R}^{3} $ such that
\begin{itemize}
\item $ \frac{1}{2}(|r|+|s|) \leq \varrho(r,s,t) \leq \max \{ |r|, |s|, |t| \} $ and, in particular, the unit sphere contains the line segment $ [(1,1,-1),(1,1,1)] $,
\item $ \varrho(r',s',t') \geq \varrho(r,s,t) $ for $ 0 \leq r \leq r', 0 \leq s \leq s', 0 \leq t \leq t' $,
\item $ \varrho(r,s,t') > \varrho(r,s,t) $ for $ 0 < r < s $ and $ 0 < t < t' $.
\end{itemize}
An example provided in \cite{kurka1} is the norm given by
$$ B_{(\mathbb{R}^{3}, \varrho)} = \mathrm{co} \, \Big( \{ (\pm 1, \pm 1, \pm 1) \} \cup \sqrt{2} B \Big), $$
where $ B $ stands for the Euclidean unit ball of $ \mathbb{R}^{3} $.

For all $ X \in \mathcal{SE}(C([0, 1])) $, we define (considering $ Q^{X}_{0} = 0 $)
$$ \sigma^{X}(z) = \Big( \sum_{n=1}^{\infty} \frac{1}{2^{n+2}} \big\Vert Q^{X}_{n}z - Q^{X}_{n-1}z \big\Vert^{2} \Big)^{1/2}, \quad z \in Z, $$
and
$$ \opnorm{z}^{X} = \varrho \big( \Vert z \Vert, \Vert z \Vert^{X}, \sigma^{X}(z) \big), \quad z \in Z. $$

\begin{claim} \label{multiclaim}
The following properties take place:

{\rm (i)} $ \sigma^{X} $ is a strictly convex seminorm on $ Z $,

{\rm (ii)} $ \opnorm{\cdot}^{X} $ is a norm on $ Z $,

{\rm (iii)} $ \sigma^{X}(z) \leq \Vert z \Vert $,

{\rm (iv)} $ \Vert z \Vert \leq \opnorm{z}^{X} \leq 2\Vert z \Vert $,

{\rm (v)} $ \sigma^{X}(Q^{X}_{n}z) \leq \sigma^{X}(z) $,

{\rm (vi)} $ \opnorm{Q^{X}_{n}z}^{X} \leq \opnorm{z}^{X} $,

{\rm (vii)} the function $ (X, z) \mapsto \sigma^{X}(z) $ is Borel,

{\rm (viii)} the function $ (X, z) \mapsto \opnorm{z}^{X} $ is Borel.
\end{claim}

\begin{proof}
(i) As the range of $ Q^{X}_{n} - Q^{X}_{n-1} $ is one-dimensional, there is $ z^{*}_{n} \in Z^{*} $ such that $ \Vert z^{*}_{n} \Vert \leq 2 $ and $ \Vert Q^{X}_{n}z - Q^{X}_{n-1}z \Vert = |z^{*}_{n}(z)| $ for every $ z \in Z $. Let us consider
$$ T : Z \to \ell_{2}, \quad z \mapsto \Big( \frac{1}{2^{(n+2)/2}} z^{*}_{n}(z) \Big)_{n=1}^{\infty}. $$
Then
$$ \sigma^{X}(z) = \Vert Tz \Vert, \quad z \in Z. $$
At the same time, $ T $ is injective (if $ Tz = 0 $, then $ Q^{X}_{n}z - Q^{X}_{n-1}z = 0 $ for all $ n $, and so $ Q^{X}_{n}z = 0 $ for all $ n $). Therefore, (i) follows from strict convexity of $ \ell_{2} $.

(ii) Using (i) and a property of $ \varrho $,
\begin{align*}
\opnorm{u+v}^{X} & = \varrho \big( \Vert u+v \Vert, \Vert u+v \Vert^{X}, \sigma^{X}(u+v) \big) \\
 & \leq \varrho \big( \Vert u \Vert + \Vert v \Vert, \Vert u \Vert^{X} + \Vert v \Vert^{X}, \sigma^{X}(u) + \sigma^{X}(v) \big) \\
 & \leq \varrho \big( \Vert u \Vert, \Vert u \Vert^{X}, \sigma^{X}(u) \big) + \varrho \big( \Vert v \Vert, \Vert v \Vert^{X}, \sigma^{X}(v) \big) \\
 & = \opnorm{u}^{X} + \opnorm{v}^{X}.
\end{align*}
The verification of $ \opnorm{\lambda z}^{X} = |\lambda| \opnorm{z}^{X} $ is similar.

(iii) We have
$$ \sigma^{X}(z)^{2} = \sum_{n=1}^{\infty} \frac{1}{2^{n+2}} \big\Vert Q^{X}_{n}z - Q^{X}_{n-1}z \big\Vert^{2} \leq \sum_{n=1}^{\infty} \frac{1}{2^{n+2}} \cdot (2 \Vert z \Vert)^{2} = \Vert z \Vert^{2}. $$

(iv) Using (I), (iii) and a property of $ \varrho $, we obtain
\begin{align*}
\Vert z \Vert & \leq \frac{1}{2} \big( \Vert z \Vert + \Vert z \Vert^{X} \big) \leq \varrho \big( \Vert z \Vert, \Vert z \Vert^{X}, \sigma^{X}(z) \big) \\
 & \leq \max \big\{ \Vert z \Vert, \Vert z \Vert^{X}, \sigma^{X}(z) \big\} = \Vert z \Vert^{X} \leq 2\Vert z \Vert.
\end{align*}

(v) We have
$$ \sigma^{X}(Q^{X}_{m}z) = \Big( \sum_{n=1}^{m} \frac{1}{2^{n+2}} \big\Vert Q^{X}_{n}z - Q^{X}_{n-1}z \big\Vert^{2} \Big)^{1/2} \leq \sigma^{X}(z). $$

(vi) Using (v) and a property of $ \varrho $, we obtain
\begin{align*}
\opnorm{Q^{X}_{n}z}^{X} & = \varrho \big( \Vert Q^{X}_{n}z \Vert, \Vert Q^{X}_{n}z \Vert^{X}, \sigma^{X}(Q^{X}_{n}z) \big) \\
 & \leq \varrho \big( \Vert z \Vert, \Vert z \Vert^{X}, \sigma^{X}(z) \big) = \opnorm{z}^{X}.
\end{align*}

(vii) This follows from (IV) and the definition of $ \sigma^{X} $.

(viii) This follows from (V), (vii) and the definition of $ \opnorm{\cdot}^{X} $.
\end{proof}

\begin{claim} \label{claim31}
We have $ \opnorm{Jx}^{X} = \Vert x \Vert $ for $ x \in X \in \mathcal{SE}(C([0, 1])) $. In particular, $ (Z, \opnorm{\cdot}^{X}) $ contains an isometric copy of $ X $.
\end{claim}

\begin{proof}
Let us assume that $ \Vert x \Vert = 1 $. Note that $ \Vert Jx \Vert^{X} = \Vert Jx \Vert = \Vert x \Vert = 1 $ due to (II). At the same time, $ \sigma^{X}(Jx) \leq \Vert Jx \Vert = 1 $ due to Claim~\ref{multiclaim}(iii). Since the unit sphere $ S_{(\mathbb{R}^{3}, \varrho)} $ contains the line segment $ [(1,1,-1),(1,1,1)] $, we obtain $ \opnorm{Jx}^{X} = 1 = \Vert x \Vert $.
\end{proof}

\begin{claim} \label{claim32}
Let $ X \in \mathcal{SE}(C([0, 1])) $ and let $ [u, v] $ be a non-degenerate line segment in $ Z $ such that $ \opnorm{\cdot}^{X} $ is constant on $ [u, v] $. Then the segment $ [u, v] $ is contained in $ JX $.
\end{claim}

\begin{proof}
It is enough to show that $ w = \frac{1}{2}(u+v) \in JX $ (the argument can be repeated for any subsegment of $ [u, v] $). Assume the opposite, i.e., $ w \notin JX $. By Claim~\ref{multiclaim}(i),
$$ \sigma^{X}(w) < \frac{1}{2}\big( \sigma^{X}(u) + \sigma^{X}(v) \big). $$
At the same time, $ \Vert w \Vert < \Vert w \Vert^{X} $ by (I) and (II), and a property of $ \varrho $ provides
$$ \varrho\Big( \Vert w \Vert, \Vert w \Vert^{X}, \frac{1}{2}\big( \sigma^{X}(u) + \sigma^{X}(v) \big) \Big) > \varrho \big( \Vert w \Vert, \Vert w \Vert^{X}, \sigma^{X}(w) \big) = \opnorm{w}^{X}. $$
The computation
\begin{align*}
\frac{1}{2}\big( & \opnorm{u}^{X} + \opnorm{v}^{X} \big) \\
 & = \frac{1}{2} \Big( \varrho \big( \Vert u \Vert, \Vert u \Vert^{X}, \sigma^{X}(u) \big) + \varrho \big( \Vert v \Vert, \Vert v \Vert^{X}, \sigma^{X}(v) \big) \Big) \\
 & \geq \varrho \Big( \frac{1}{2} \big( \Vert u \Vert + \Vert v \Vert \big), \frac{1}{2} \big( \Vert u \Vert^{X} + \Vert v \Vert^{X} \big), \frac{1}{2}\big( \sigma^{X}(u) + \sigma^{X}(v) \big) \Big) \\
 & \geq \varrho\Big( \Vert w \Vert, \Vert w \Vert^{X}, \frac{1}{2}\big( \sigma^{X}(u) + \sigma^{X}(v) \big) \Big) > \opnorm{w}^{X} 
\end{align*}
concludes the proof.
\end{proof}

\begin{claim} \label{claim33}
{\rm (1)} $ (Z, \opnorm{\cdot}^{X}) $ is isometrically universal for all separable Banach spaces if and only if $ X $ has the same property.

{\rm (2)} $ (Z, \opnorm{\cdot}^{X}) $ is strictly convex if and only if $ X $ is strictly convex.
\end{claim}

\begin{proof}
We check only the implication \textquotedblleft $ \Rightarrow $\textquotedblright {} in (1), since other implications follow from Claims \ref{claim31} and \ref{claim32}. Let us denote
$$ \Delta = \{ 0, 1 \}^{\mathbb{N}}, \quad \Delta(i) = \{ \gamma \in \Delta : \gamma(1) = i \}, \quad i = 0, 1, $$
$$ H = C(\Delta), \quad H(i) = \{ h \in H : \gamma \notin \Delta(i) \Rightarrow h(\gamma) = 0 \}, \quad i = 0, 1. $$
Assume that there is an isometry $ I : H \to (Z, \opnorm{\cdot}^{X}) $ and denote
$$ z = I (\mathbf{1}_{\Delta(0)}). $$
We claim that the space $ JX $ (and therefore the space $ X $ by Claim~\ref{claim31}) is universal, showing that $ I $ maps $ H(1) $ into $ JX $.

Given $ h \in H(1) $ such that $ \Vert h \Vert \leq 1 $, we observe that $ \Vert \mathbf{1}_{\Delta(0)} \Vert = \Vert \mathbf{1}_{\Delta(0)} \pm h \Vert = 1 $, and so $ \opnorm{z}^{X} = \opnorm{z \pm Ih}^{X} = 1 $. By Claim~\ref{claim32}, we have $ Ih \in JX $.
\end{proof}

Our last claim is a statement similar to \cite[Theorem~17]{dodos} and \cite[Theorem~5.19]{dodostopics}.

\begin{claim} \label{claim34}
The set
$$ \mathcal{R} = \big\{ (X, Y) \in \mathcal{SE}(C([0, 1]))^{2} : Y \textrm{ is isometric to } (Z, \opnorm{\cdot}^{X}) \big\} $$
is analytic.
\end{claim}

\begin{proof}
Let $ s_{1}, s_{2}, \dots $ be a dense sequence in $ Z $. Let us recall that the function $ (X, z) \mapsto \opnorm{z}^{X} $ is Borel by Claim~\ref{multiclaim}(viii). Therefore, the set $ \mathcal{R} $ is a projection of a Borel set in $ \mathcal{SE}(C([0, 1]))^{2} \times Z^{\mathbb{N}} $, as
\begin{eqnarray*}
(X, Y) \in \mathcal{R} & \Leftrightarrow & \exists (z_{1}, z_{2}, \dots ) \in Z^{\mathbb{N}} : \\
 & & \bigg[ \Big( \forall k \in \mathbb{N} \; \forall l \in \mathbb{N} \; \exists n \in \mathbb{N} : \Vert s_{k} - z_{n} \Vert < \frac{1}{l} \Big) \\
 & & \& \; \Big( \forall m \in \mathbb{N} \; \forall \gamma_{1}, \dots, \gamma_{m} \in \mathbb{Q} : \\
 & & \opnorm[\Big]{\sum_{n=1}^{m} \gamma_{n} z_{n}}^{X} = \Big\Vert \sum_{n=1}^{m} \gamma_{n} d_{n}(Y) \Big\Vert \Big) \bigg]
\end{eqnarray*}
where a sequence of mappings $ d_{1}, d_{2}, \dots : \mathcal{F}(C([0, 1])) \to C([0, 1]) $ is provided by Theorem~\ref{kurryl}.
\end{proof}

Let us finish the proof of Theorems \ref{thmmain1} and \ref{thmmain2}. Depending on the theorem we want to prove, let $ P $ denote the property of being not isometrically universal for all separable Banach spaces or the property of being strictly convex.

By \cite[Theorem 1.2]{kurka1}, the theorems have been already proved under the assumption that the members of $ \mathcal{C} $ have a monotone basis. Therefore, it is sufficient to show the following.

\emph{Let $ \mathcal{C} $ be an analytic set of separable Banach spaces which satisfy $ P $. Then there exists an analytic set $ \mathcal{C}' $ of Banach spaces which satisfy $ P $ such that every member of $ \mathcal{C'} $ has a monotone basis and such that an isometric copy of every member of $ \mathcal{C} $ is contained in a member of $ \mathcal{C'} $.}

Given such $ \mathcal{C} $, the set
$$ \mathcal{C}' = \big\{ Y \in \mathcal{SE}(C([0, 1])) : Y \textrm{ isometric to } (Z, \opnorm{\cdot}^{X}) \textrm{ for some } X \in \mathcal{C} \big\} $$
is analytic by Claim~\ref{claim34}, since it is a projection of the analytic set $ \mathcal{R} \cap (\mathcal{C} \times \mathcal{SE}(C([0, 1]))) $.

Let us check that $ \mathcal{C}' $ works. By Claim~\ref{claim33}, every $ Y \in \mathcal{C}' $ satisfies $ P $. By Claim~\ref{multiclaim}(vi), every $ Y \in \mathcal{C}' $ has a monotone basis. Finally, every $ X \in \mathcal{C} $ is contained in some $ Y \in \mathcal{C}' $ by Claim~\ref{claim31}.

\end{document}